\documentclass[a4paper,10pt]{article}
\usepackage[utf8]{inputenc}
\usepackage[english]{babel}
\usepackage{comment,xcomment,amssymb,amsmath,color}
\usepackage{array}
\newcolumntype{L}{>{$}l<{$}} 
\newcolumntype{C}{>{$}c<{$}} 

\usepackage{rotating,xypic,array}
\usepackage{latexsym,mathtools,amsthm}
\usepackage{booktabs}  
\usepackage{paralist}
\usepackage{longtable}
\usepackage{url} 
\usepackage[section]{placeins}
\usepackage{multirow,rotating,graphicx}
\usepackage[a4paper]{geometry}
\usepackage[colorlinks, citecolor=blue]{hyperref}
\usepackage{xspace}
\usepackage{pdflscape}
\makeatletter
\def\namedlabel#1#2{\begingroup
   \def\@currentlabel{#2}%
   \label{#1}\endgroup
}
\makeatother

\newtheorem{theorem}{Theorem}[section]
\newtheorem{lemma}[theorem]{Lemma}

\newtheorem{proposition}[theorem]{Proposition}
\newtheorem*{theoremA}{Theorem A}
\newtheorem*{theoremB}{Theorem B}

\theoremstyle{definition}

\newtheorem{example}[theorem]{Example}

\theoremstyle{remark}

\newcommand{\abs}[1]{\left\vert#1\right\vert}
\newcommand{\R}{\mathbb{R}}
\newcommand{\lie}[1]{\mathfrak{#1}}     
\newcommand{\g}{\lie{g}}
\newcommand{\Z}{\mathbb{Z}}

\newcommand{\hook}{\lrcorner\,}

\newcommand{\Span}[1]{\operatorname{Span}\left\{#1\right\}}

\newcommand{\tran}[1]{\hspace{.2mm}\prescript{t\hspace{-.5mm}}{}{#1}}
\DeclareMathOperator{\ric}{ric}  

\DeclareMathOperator{\Aut}{Aut}

\DeclareMathOperator{\diag}{diag}
\DeclareMathOperator{\Der}{Der}

\DeclareMathOperator{\ad}{ad}

\DeclareMathOperator{\Tr}{tr}

\newcolumntype{C}{>{$}c<{$}}
\newcolumntype{L}{>{$}l<{$}}
\newcolumntype{R}{>{$}r<{$}}

\newcolumntype{C}{>{$}c<{$}}
\newcolumntype{L}{>{$}l<{$}}
\newcolumntype{R}{>{$}r<{$}}

\newcounter{rowno}
\usepackage[backend=biber]{biblatex}
\begin{document}
\title{The Ricci-flatness that lurks in weight}
\author{Diego Conti}
\maketitle

\begin{abstract}
We introduce two constructions to obtain left-invariant Ricci-flat pseudo-Riemannian metrics on nilpotent Lie groups, one based on gradings, the other on filtrations, both depending on the combinatorics of the set of weights.

As an application, we show that every nilpotent Lie algebra of dimension up to $7$ and every nice nilpotent Lie algebra of dimension up to $9$ admit an indefinite Ricci-flat metric.
\end{abstract}

\renewcommand{\thefootnote}{\fnsymbol{footnote}}
\footnotetext{\emph{MSC class 2020}: \emph{Primary} 53C25; \emph{Secondary} 17B70 22E25 53C50}
\footnotetext{\emph{Keywords}: Ricci-flat metric, homogeneous metric, pseudo-Riemannian metric, nilpotent Lie algebra, graded Lie algebra.}
\renewcommand{\thefootnote}{\arabic{footnote}}

This paper is part of a broader programme to study homogeneous indefinite Einstein metrics. In light of the proof of the Aleksveesky conjecture by B\"ohm and Lafuente \cite{Böhm_Lafuente_2023}, a natural first step in this direction is understanding left-invariant Einstein metrics on solvable Lie groups. Contrary to the Riemannian case (see \cite{Alekseevskii_Kimel’fel’d_1975}), homogeneous Ricci-flat metrics on indefinite signature need not be flat. One therefore has to account for the existence of these metrics.

A first, natural question is whether every solvable Lie group admits a left-invariant Ricci-flat metric of some signature. The answer is negative: for instance, in dimension three all Lorentzian left-invariant metrics and their curvature are classified (\cite{Cordero_Parker}), and not all solvable Lie algebras have a Ricci-flat metric. The question remains open if ``solvable'' is replaced with ``nilpotent''.

The problem is normally studied at the Lie algebra level: rather than a left-invariant metric on a Lie group, one considers a metric on its Lie algebra $\g$, i.e. a nondegenerate scalar product, and the Levi-Civita connection, curvature and Ricci tensors are viewed as defined on $\g$ by restricting the corresponding objects relative to the associated left-invariant metric.

Existence of a Ricci-flat metric is known for some special classes: nilpotent Lie algebras with large center (\cite{Yan_Deng_2023});
filiform Lie algebras of rank two and almost abelian Lie algebras (\cite{Xiang_Yan_2021}); step-two nilpotent Lie algebras attached to graph, nilpotent Lie algebras of dimension up to $6$ and nilpotent Lie algebras of dimension $7$ admitting a nice basis (\cite{Conti_delBarco_Rossi_2021}). If one restricts to Lorentzian signature, not every nilpotent Lie algebra admits a Ricci-flat metric; classifications in step two and three are given in  \cite{Guediri_BinAsfour_2014,Boucetta_MohamedMM_2021}.

The techniques used to construct these indefinite Ricci-flat metrics differ considerably from those that are used in the study of the Ricci operator in positive-definite signature, which typically exploit the fact that Ricci-flat metrics are critical points of the scalar curvature functional, with constrained critical points corresponding to the weaker conditions on the Ricci operator that are of interest in the Riemannian case (see \cite{Nikolayevsky_2011,Arroyo_Lafuente_2020}). The essential difficulty in trying to adapt these methods is that in indefinite signature the scalar curvature functional fails to be convex.

By contrast,
most of the known Ricci-flat metrics on nilpotent Lie algebras are obtained either by choosing an appropriate ansatz and solving the polynomial equations in the metric coefficients corresponding to Ricci-flatness, or by using some version of the double extension procedure introduced in \cite{Medina_Revoy_1985} for $\ad$-invariant metrics. In \cite{Conti_delBarco_Rossi_2021}, a different construction was introduced, based on combinatorial objects called ``arrow-breaking involutions''. This method only applies to the class of nilpotent Lie algebras that admits a so-called \emph{nice basis}, which is a special kind of basis adapted to both the lower and upper central series. The idea is that it is sometimes possible to fix such a basis $\{e_i\}$ and then take a metric such that $e_i$ is orthogonal to $e_j$ unless $i=\sigma(j)$ for some order two permutation $\sigma$, in such a way that Ricci-flatness descends solely from the combinatorics of the nonvanishing structure constants (i.e.,  the  set of triples $(i,j,k)$ for which $[e_i,e_j]$ has a component along $e_k$), irrespective of the value of said structure constants. In particular, no polynomial equation needs to be solved. The main shortcoming of this approach is that carrying a nice basis is a restrictive condition on a nilpotent Lie algebra (increasingly so as the dimension grows). Additionally, the construction does not apply to all the  nilpotent Lie algebras that admit a nice basis.

The present paper is aimed at overcoming these limitations, whilst adopting a similar philosophy. Indeed, we introduce two constructions. The first replaces nice bases with the more general class of bases which diagonalize some split torus in the group of automorphisms; equivalently, we consider bases adapted to a grading. In Theorem~\ref{thm:foad} we give  sufficient conditions on  a grading that guarantees the existence of a Ricci-flat metric. Like in the arrow-breaking construction, these conditions  abstract from the structure constants and only make use of the grading; additionally, they do not require fixing a special basis, as the existence of a nontrivial split torus of derivations is an intrinsic property of the Lie algebra which can be determined algorithmically (see ~\cite{Kochetov_2009, Hakavuori_Kivioja_Moisala_Tripaldi_2022}). However, this construction does not apply to all Lie algebras. Most notably, it does not apply to characteristically nilpotent Lie algebras, namely the nilpotent Lie algebras all of whose derivations are nilpotent (which are not as scarce as one may expect, see e.g. \cite{Goze_Hakimjanov_1994}).

The second construction considers filtrations rather than gradings, and applies more often, as is to be expected if one considers that  every nilpotent Lie algebra admits nontrivial filtrations. As a filtered counterpart to Theorem~\ref{thm:foad}, in Theorem~\ref{thm:ffoad} we give  sufficient conditions on a filtration that imply existence of a Ricci-flat metric. As in the other case, these conditions do not depend on the value of the structure constants, and finding the metric does not require solving polynomial equations. On the other hand, determining filtrations is more complicated than computing a maximal split torus in the group of automorphisms. In Section~\ref{sec:enumeratefiltrations} we describe a simple algorithm that for a fixed basis on a Lie algebra determines filtrations satisfying the conditions of Theorem~\ref{thm:ffoad} such that each layer  is spanned by basis elements. We also prove that some Lie algebras do not carry filtrations of this type. However, we do not have an algorithm which can determine conclusively if Theorem~\ref{thm:ffoad} can be applied to a given Lie algebra.

Thus, we view the two constructions as complementary. Their relative effectiveness can be measured by testing them on  low-dimensional nilpotent Lie algebras. We recall that nilpotent Lie algebras are classified up to dimension $7$ (\cite{Gong_1998}), and nilpotent Lie algebras with a nice basis up to $9$ (\cite{Conti_Rossi_2019a}); in \cite{Conti_delBarco_Rossi_2021} the Lie algebras in these classifications were proved to admit a Ricci-flat metric up to dimensions $6$ and $7$, respectively. Using our constructions, we find:
\begin{theoremA}
\label{thm:dim7}
Every nilpotent Lie algebra of dimension $\leq 7$ has a Ricci-flat metric.
\end{theoremA}
\begin{theoremB}
\label{thm:dim89}
Every nilpotent Lie algebra of dimension $\leq 9$ admitting a nice basis has a Ricci-flat metric.
\end{theoremB}
Our proof is based on case-by-case computations with an ad-hoc computer program (\cite{skoll}): for each Lie algebra, we determine either a grading satisfying Theorem~\ref{thm:foad}, a filtration satisfying Theorem~\ref{thm:ffoad}, or an explicit Ricci-flat metric when both methods fail. More precisely, in the situation of Theorem~A, existence of a Ricci-flat metric can be proved using filtrations for all Lie algebras except two, whilst the construction with gradings fails in $22$ cases. For nice Lie algebras of dimension $8$, filtrations apply in all but six cases, and gradings in all but $93$. In dimension $9$ it becomes difficult to determine exactly which Lie algebras  admit a filtration satisfying the hypotheses of Theorem~\ref{thm:ffoad}. Nevertheless, we were able to produce filtrations or gradings implying the existence of a Ricci-flat metric for $6851$ out of the of $6882$ nice nilpotent Lie algebras of dimension $9$, leaving $31$ for which a Ricci-flat metric had to be produced explicitly.

\section{Gradings, filtrations and the Ricci-flat condition}
In this section we introduce two related but different constructions to produce Ricci-flat metrics on nilpotent Lie algebras. Our aim is to prove existence of a Ricci-flat metric based on the existence of a grading or filtration of a particular class, regardless of the structure constants. The construction is based on the formula for the Ricci tensor
\[\ric(v,w)=\frac12 g(dv^\flat,dw^\flat)-\frac12g(\ad v,\ad w),\]
which holds on any nilpotent Lie algebra (see e.g. \cite{Besse,Conti_Rossi_2019}; for arbitrary Lie algebras, two extra terms appear depending on the one-form $X\mapsto \Tr \ad X$ and the Killing form). Specifically, we construct metrics such that $\ad\g\subset\g^*\otimes\g$ and $d\g^*$ are isotropic subspaces, i.e. the metric restricts trivially. The metrics will take the antidiagonal form
\begin{equation}
\label{eqn:antidiagonal}
\sum_{i=1}^n e^i\otimes e^{\hat i}=\sum_{i=1}^n e^i\otimes e^{n+1-i},
\end{equation}
relative to some basis $e^i$ of $\g^*$. Throughout the paper, $\{e^i\}$ will denote a basis of $\g^*$, $\{e_i\}$ the dual basis of $\g$, and we will write $\hat i$ for $n+1-i$.

Our construction is based on two lemmas which are best formulated in terms of graded vector spaces. Let $A$ be a torsion-free abelian group; a grading over $A$ of a vector space $V$ is a decomposition \mbox{$V=\bigoplus_{\alpha\in A}V_\alpha$}. If $V_\alpha$ is not the trivial subspace,  $V_\alpha$ is called a \emph{layer} and $\alpha$ is a \emph{weight}. We will only consider finite-dimensional vector spaces, so it will be no loss of generality to assume that $A$ is finitely generated. There  would be no loss of generality by assuming that $A$ is a subgroup of $\R$, but we will not do so as it would confuse matters rather than simplifying.

Given a graded vector space $V=\bigoplus_{\alpha\in A}V_\alpha$, we say that the ordered basis  $e_1,\dotsc, e_n$ is \emph{adapted} to the grading if each $V_\alpha$ is spanned by  elements of the basis. To an adapted basis, we associate the \emph{weight sequence} $w_1,\dotsc, w_n$ such that  $e_i$ has weight $w_i$. Thus, each $\alpha\in A$ appears $\dim V_\alpha$ times in the weight sequence; we will say that $\dim V_{\alpha}$ is the \emph{multiplicity} of the weight $\alpha$. We will write  $V_\alpha\wedge V_\beta$ for the image in $\Lambda^2 V$ of $V_\alpha\otimes V_\beta$. Therefore, $\Lambda^2V$ is a direct sum of the spaces $V_\alpha\wedge V_\beta$, as $\{\alpha,\beta\}$ varies among unordered pairs of weights.

Recalling that the index of a scalar product of signature $(p,q)$ is defined as the difference $p-q$, it is clear that scalar products which can be written in the form \eqref{eqn:antidiagonal} are exactly those of index $0$ or $1$.

We will need the following elementary result.
\begin{lemma}
\label{lemma:isotropic}
Let $V$ be a vector space, $\dim V\neq 2$, and let $F$ be an element of $\Lambda^2V^*$. Then there is a scalar product on $V$ of index $0$ or $1$ such that the two-form $F$ is isotropic. Furthermore, if $V$ is the direct sum $U\oplus W$ of subspaces of the same dimension and $F\in U^*\wedge W^*$, the scalar product can be chosen in such a way that  $U$ and $W$ are isotropic.
\end{lemma}
\begin{proof}
We first prove the second part. Suppose $V=U\oplus W$, with $\dim U=\dim W\geq 2$, and such that $F\in U^*\wedge W^*$. We can associate to $F$ a linear map
\[f_F\colon U\to W^*, \quad u\mapsto F(u,\cdot).\] Let $f\colon U\to W^*$ be any isomorphism; a scalar product is induced,
\[g(u+w,u'+w')=f(u)w'+f(u')w,\]
with obvious notation. We have
\[g(F,F)=\Tr((F^\sharp)^2), \quad g(F^\sharp(v),v')=F(v,v').\]
Thus
\[F^{\sharp}|_U= f^{-1}\circ f_F\colon U\to U, \quad F^\sharp|_W=-\tran f^{-1}\circ \tran f_F.\]
It follows that
\[g(F,F)=2\Tr((f^{-1}\circ f_F)^2).\]
It is now clear that, $f_F$ being fixed, we can choose $f$ in such a way that $\Tr((f^{-1}\circ f_F)^2)$ is zero. Indeed, if $f_F$ has rank $r\geq 3$, we can assume that, relative to some basis $e_1,\dotsc, e_n$,
\[f^{-1}\circ f_F=e^1\otimes e_2+\dots + e^{r-1}\otimes e_r+ e^r\otimes e_1;\]
if $f_F$ has rank $2$, we can assume it takes the form
\[e^1\otimes (e_1+e_2)+e^2\otimes (-e_1+e_2).\]
If $f_F$ has rank $1$,  we can assume it takes the form $e^1\otimes e_2$, as $U$ has dimension at least two. In each case, the square has trace zero; thus, the second part of the statement is proved.

For the first part, observe that if $V$ has even dimension greater than two, it is always possible to write $V=U\oplus W$, with $\dim U=\dim W\geq 2$, and such that $F\in U^*\wedge W^*$, so we can construct the metric as above.

If $V$ has odd dimension greater than three, we can write $V=U\oplus W\oplus\R$, with $\dim U=\dim W\geq 2$ and $F\in U^*\wedge W^*$. We can then construct a neutral metric on $U\oplus W$ as above and extend it to $V$ by declaring that the factor $\R$ is orthogonal to $U\oplus W$, with a positive definite metric; this metric has index one and makes $F$ isotropic.

Finally, if $V$ has dimension three, we can fix a basis $e_1,e_2,e_3$ such that $F$ is a multiple of  $e^{12}$, and choose $g=e^1\odot e^3+e^2\otimes e^2$.
\end{proof}

Our strategy to obtain Ricci-flatness is to consider a class of metrics such that $\ad \g$ is isotropic and all scalar products $g(de^i,de^j)$ vanish except possibly $g(de^n,de^n)$; choosing the metric as in Lemma~\ref{lemma:isotropic}, with $F=de^n$, will then imply that $d\g^*$ is also isotropic. Specifically, we will use the following:
\begin{lemma}
\label{lemma:multiplicities}
Let $V$ be a graded vector space with weight sequence $w_1,\dotsc, w_n$. Suppose that
\begin{compactenum}
\item  if $w_i+w_{\hat i}=w_n$ and $w_i\neq w_{\hat i}$, then at least one of $w_i,w_{\hat i}$ has multiplicity at least two;
\item  if $w_i+w_{\hat i}=w_n$ and $w_i= w_{\hat i}$ with $i$ and $\hat i$ distinct, then $w_i$ has multiplicity greater than two.
\end{compactenum}
Then for every element $F$ of $\sum_{\alpha+\beta=w_n}V_{\alpha}\wedge V_\beta$ there is an adapted basis $e_1,\dotsc, e_n$ with weight sequence $w_1,\dotsc, w_n$ such that $F$ is isotropic relative to the antidiagonal metric $\sum e^i\otimes e^{\hat i}$.
\end{lemma}
\begin{proof}
For each weight $\alpha$, we need  to choose a basis  of the layer $V_\alpha$. If $w_n-\alpha$ is not a weight, the basis can be chosen arbitrarily, since $F$ has no component along $V_\alpha\wedge V$.

If $w_n-\alpha$ is a weight $\beta$, there are two cases.
\begin{compactenum}
\item If $\alpha=\beta$, partition the set of indices $1\leq j\leq n$ corresponding to the weight $\alpha$ as the disjoint union of
\[I=\{i\mid 1\leq i\leq n, \alpha=w_i=w_{\hat i}\}, \quad J=\{j\mid 1\leq j\leq n, \alpha=w_j\neq w_{\hat j}\}.\]
By hypothesis, either $I$ contains only one element, or $\abs{I}+\abs{J}>2$. Write $V_\alpha=U\oplus W$, where $\dim U=\abs{I}$, $\dim W=\abs{J}$. If $U$ has dimension two, we can choose $U$ and $W$ so that the restriction of $F$ to $U$ is zero, as $V_\alpha$ has dimension at least three. Otherwise, we can apply Lemma~\ref{lemma:isotropic} and construct a metric on $U$ of index $0$ or $1$ such that $F$ is isotropic. We can therefore choose the basis of $V_\alpha$ by assuming that $\{e_i\mid i\in I\}$ is a basis of $U$ and  $\{e_j\mid j\in J\}$ is a basis of $W$ such that $F|_U$ is isotropic for the antidiagonal metric. Then the component of $F$ in $\Lambda^2 V_{\alpha}$ is isotropic.
\item If $\alpha\neq\beta$, consider the sets
\[I=\{i\mid 1\leq i\leq n, w_i=\alpha, w_{\hat i}=\beta\}, \quad J=\{j\mid 1\leq j\leq n, \text{ either }w_j=\alpha\text{ or } w_{\hat j}=\beta \text{ but not both}\}.\]
We need to choose the elements $e_i, e_{\hat i}$ for $i\in I$ in such a way that the restriction of $F$ to
\[K=\Span{e_i, e_{\hat i}\mid i\in I}\]
is isotropic; the basis elements $e_j$ for $j\in J$ can then be chosen arbitrarily.

If $I$ is empty, there is nothing to prove. If $I$ contains exactly one element $i$, by the hypothesis one of $\alpha$, $\beta$ has multiplicity at least two, so we can choose  $e_i$ in $V_\alpha$ or $e_{\hat i}$ in $V_\beta$ in such a way that $F|_K$ is zero. If $I$ has two or more elements, we can fix subspaces $U\subset V_\alpha$, $W\subset V_\beta$, both of dimension $\abs{I}$, and then apply Lemma~\ref{lemma:isotropic}
to obtain a neutral metric on $U\oplus W$ where $U$, $W$ and $F$ are isotropic. We can then choose bases $\{e_i\mid i\in I\}$ of $U$ and $\{e_{\hat i}\mid i\in I\}$ of $W$ so that the metric becomes antidiagonal.
\end{compactenum}
\end{proof}

We will apply Lemma~\ref{lemma:multiplicities} to two different situations. In the first situation, we consider a graded Lie algebra $\g$; this means that the vector space $\g$ is endowed with a grading $\g=\bigoplus_{\alpha\in A}\g_\alpha$ such that
\[[\g_\alpha,\g_\beta]\subset\g_{\alpha+\beta}.\]
Our sufficient condition for the existence of a Ricci-flat metric will be expressed in terms of a weight sequence $w_1,\dotsc, w_n$; notice that we do not assume that equal terms in the weight sequence appear consecutively.
\begin{theorem}
\label{thm:foad}
Let $\g$ be a nilpotent Lie algebra graded over $A$; suppose that there is a weight sequence $w_1,\dotsc, w_n$ satisfying
\begin{compactenum}[(G1)]
\item \label{G1} if $w_i+w_j=w_k$ and $i\neq j$ then $i,j<k$;
\item \label{G2} if $w_i+w_{\hat i}=w_j$, then $j=n$;
\item \label{G3} if $w_i+w_{\hat i}=w_n$ and $w_i\neq w_{\hat i}$, then one of $w_i,w_{\hat i}$ has multiplicity at least two;
\item \label{G4} if $w_i+w_{\hat i}=w_n$ and $w_i= w_{\hat i}$, then either $i=\hat i$ or $w_i$ has multiplicity greater than two;
\item \label{G5} if $w_i+w_j$ is a weight and $w_{\hat i}+w_{\hat j}$ is a weight, then $w_j=w_{\hat i}$ and $w_i=w_{\hat j}$.
\end{compactenum}
Then $\g$ has a Ricci-flat metric of the form $\sum e^i\otimes e^{\hat i}$, where $e_i$ is in the layer of weight $w_i$.
\end{theorem}
\begin{proof}
As a first step, observe that relative to a metric of the form $\sum e^i\otimes e^{\hat i}$, where $\{e_i\}$ is any adapted basis with weight sequence $w_1,\dotsc, w_n$, the space $\ad\g$ is isotropic. Indeed,
\begin{equation}
 \label{eqn:adgincreasing}
 \ad\g\subset\Span{e^i\otimes e_{k}\mid w_i+w_j=w_k \text{ for some $j$}}\subset \Span{e^i\otimes e_k\mid i<k}.
\end{equation}
Since every element $e^i\otimes e_k$ has scalar product zero with every $e^j\otimes e_l$ except $e^{\hat i}\otimes e_{\hat k}$, which does not appear in the right hand side of \eqref{eqn:adgincreasing}, we see that $\ad\g$ is indeed isotropic.

We therefore need to show that the adapted basis $\{e_i\}$ can be chosen in such a way that the space $d\g^*$ is also isotropic. We first show that $g(dx,dy)=0$ if one of $x,y$ has weight $w_i\neq w_n$. Indeed, suppose that $x\in\g_\alpha^*$, $y\in\g_\beta^*$ and $g(dx,dy)$ is not zero. By construction, $dx$ is a linear combination of elements $e^i\wedge e^j$ where $w_i+w_j=\alpha$, and $dy$ must contain some corresponding element $e^{\hat i}\wedge e^{\hat j}$ with $w_{\hat i}+w_{\hat j}=\beta$. By (G5), this implies that $w_i+w_j=w_i+w_{\hat i}$, so  by (G2)
$\alpha=w_n$ and it has multiplicity one, and similarly for $\beta$.

Thus, for any adapted basis $e_1,\dotsc, e_n$ with weight sequence $w_1,\dotsc, w_n$ we have that $d\g^*$ is isotropic if and only if so is $de^n$.
Thanks to  (G3) and (G4), it now suffices to apply Lemma~\ref{lemma:multiplicities}.
\end{proof}
The construction of Theorem~\ref{thm:foad} has the advantage that it depends on the grading in a purely combinatorial way, and gradings can be determined by computing a maximal split torus in the Lie algebra of derivations (see  Section~\ref{sec:enumerategradings}). It is therefore possible in principle to determine exactly whether a given Lie algebra has a metric of this type or not, and the computation does not depend on the basis in which the Lie algebra is given. In sight of constructing Ricci-flat metrics on nilpotent Lie algebras, however, the construction is not sufficiently general. For instance, it does not apply to Lie algebras without nontrivial gradings, or such that every grading has $0$ as a weight, which violates (G1); however, such Lie algebrs may well admit Ricci-flat metrics, and indeed we are not aware of any nilpotent Lie algebra that does not admit a Ricci-flat metric. We refer to Section~\ref{sec:enumerategradings} for other examples of Lie algebras where Theorem~\ref{thm:foad} does not apply.

With this motivation, we introduce a second construction, which makes use of
\emph{filtered} Lie algebras rather than graded. Precisely, on a Lie algebra $\g$ we will consider filtrations of $\g$ as a vector space taking  the form
\[\g=L_1\supset \dots \supset L_N,\]
with the property that $[L_i,L_j]\subset L_{i+j}$, and the convention that  $L_{k}=0$ for $k>N$. Such a filtration will be called a \emph{positive filtratio}n, to distinguish from  filtrations of the form $\g=L_{-1}\subset L_0\subset\dots$,  also considered in the literature (see \cite{Kobayashi_Nagano_1964}); we will say that $\g$ is \emph{positively filtered}.

Notice that we do not assume the inclusions to be strict. With this convention, a positively graded Lie algebra $\g=\bigoplus_{h>0} \g_h$  is positively filtered by setting $L_w=\bigoplus_{h\geq w}\g_h$. Even in the absence of a nontrivial grading, a nilpotent Lie algebra always admits positive filtrations, for instance the upper and lower central series. Conversely, the existence of a positive filtration forces the Lie algebra to be nilpotent, since taking Lie brackets in sequence gives zero after $N$ steps.

We will say that a basis $e_1,\dotsc, e_n$ is \emph{adapted} to the filtration if each $L_w$ is the span of the last $\dim L_w$ elements. Given an element of $\g$, its \emph{weight} is the largest $w$ such that $L_w$ contains the element. Given an adapted basis, we will write $w_k$ for the weight of $e_k$. Thus, if $[e_i,e_j]=e_k$, then $e_k$ is contained in $L_{w_i+w_j}$, i.e. $w_i+w_j\leq w_k$. Conversely, given a basis $\{e_i\}$ of a Lie algebra $\g$, one can give a filtration by assigning weights $w_i$ in such a way that
\begin{equation}
 \label{eqn:filtration}
w_i+w_j\leq w_k \text{ whenever } de^k(e_i,e_j)\neq0,
 \end{equation}
and then  declaring that $L_w$ is spanned by elements with weight $\geq w$.
\begin{theorem}
 \label{thm:ffoad}
Let $\g$ be a positively filtered nilpotent Lie algebra; let $e_1,\dotsc,e_n$ be an adapted basis with $e_i$ of weight $w_i$. Suppose that
\begin{compactenum}[(F1)]
\item \label{F1} $0<w_1\leq w_2\leq \dotsc \leq w_n$;
\item \label{F2} $w_i+w_{\hat i}\geq w_{n}$, $w_i+w_{\hat i}> w_{n-1}$;
\item \label{F3} if $w_i+w_{\hat i}=w_n$ and $w_i\neq w_{\hat i}$, then one of $w_i,w_{\hat i}$ has multiplicity at least two;
\item \label{F4} if $w_i+w_{\hat i}= w_n$ and $w_i= w_{\hat i}$, then either $i=\hat i$ or $w_i$ has multiplicity greater than two.
\end{compactenum}
Then the metric $\sum e^i\otimes e^{\hat i}$ is Ricci-flat.
\end{theorem}
\begin{proof}
Since $\g$ is filtered over positive integers,  whenever $e_i$ has weight $k$, $[\g,e_i]$ is contained in $L_{k+1}$. Because the basis is adapted, $L_{k+1}$ is spanned by $e_h,\dotsc, e_n$, and the fact that $e_i$ has weight $k$ implies $L_{k+1}\subset \Span{e_{i+1},\dotsc, e_n}$. This shows that $\ad\g^*\subset\Span{e^i\otimes e_j\mid i<j}$. Since the metric is antidiagonal, $\ad \g$ is isotropic.

We need to show that $d\g^*$ is also isotropic. Let $de^k$, $de^h$ be nonorthogonal two-forms;  this means that $de^k$ has a component of the form $e^{ij}$ such that $de^h$ has a component of the form $e^{\hat i\hat j}$. In other words, we have
\[de^k(e_i,e_j)\neq 0, \quad de^h(e_{\hat i}, e_{\hat j})\neq0.\]
This implies that $w_i+w_j\leq w_k$, $w_{\hat i}+w_{\hat j}\leq w_h$. Taking the sum and using (F2), we find
\[2w_{n}\leq w_i+w_{\hat i}+w_j+w_{\hat j}\leq w_k+w_h,\]
and equality holds by (F1); (F2) then implies that $h=k=n$ and $w_n=w_i+w_{\hat i}=w_j+w_{\hat j}$. Thus,  $w_{\hat i}=w_j$, $w_{\hat j}=w_i$.

If $\pi$ denotes the projection which sends each $e^{ij}$ to itself if $w_j=w_{\hat i}$ and zero otherwise, we see that $g(\pi(de^n),\pi(de^n))=g(de^n,de^n)$, and
$d\g^*$ is isotropic if and only if so is $\pi(de^n)$. Considering the grading of $\g$ as a vector space defined by assigning each $e_i$ to the layer of weight $w_i$,  Lemma~\ref{lemma:multiplicities} shows that the adapted basis can be modified so that $\pi(de^n)$ is isotropic.
\end{proof}

\section{Metrics from gradings}
\label{sec:enumerategradings}
In this section we illustrate how to apply Theorem~\ref{thm:foad} to find a Ricci-flat metric on a fixed Lie algebra.

We first recall that torsion-free gradings of a  Lie algebra are in one-to-one correspondence with split tori in $\Aut(\g)$ (see \cite[Proposition 4.1]{Kochetov_2009}; note that we are working over $\R$). Explicitly, if a split torus with Lie algebra $\lie t$ acts on $\g$ by automorphisms with weights $\alpha\colon\lie t\to\R$, the associated grading is given by the weight spaces $\g_\alpha$.

The Lie group of automorphisms of a Lie algebra is algebraic; hence, its Lie algebra $\Der\g$ admits the decomposition
\[\Der\g=\lie s\oplus \lie a \oplus \lie n,\]
where $\lie s$ is semisimple, $\lie a+\lie n$ is the radical, $\lie n$ is the nilradical, $\lie a$ is a torus commuting with $\lie s$ (see \cite{Chevalley_1947}). Moreover, we can use the trace form $(X,Y)\mapsto \langle X,Y\rangle_{\Tr}=\Tr(XY)$ to write
\[\lie a+\lie n=[\Der\g,\Der\g]^{\perp_{\Tr}}, \quad \lie n=(\Der\g)^{\perp_{\Tr}}.\]
We can therefore identify $\lie a$ by computing the radical and nilradical and then identifying a complement of the latter inside the first. The complement must be chosen so it acts in a semisimple way on $\g$, which generally requires using Jordan decomposition. Borrowing notation from \cite{Nikolayevsky_2011}, we then decompose $\lie a$ as the sum of a compact torus $\lie a_{i\R}$ and a split torus $\lie a_\R$, i.e. a torus acting with imaginary eigenvalues and one acting with real eigenvalues.

\begin{example}
Consider the Lie algebra
\[0,0,e^{12},e^{13},e^{23},e^{25}+e^{14};\]
the notation, which is adapted from \cite{Salamon_2001} and will be used throughout the paper, means that relative to a fixed coframe $e^1,\dotsc, e^6$, the forms $de^1$ and $de^2$ vanish, $de^3=e^1\wedge e^2$ and so on. A straightforward computation using e.g. \cite{skoll} shows that, writing $e_i^j$ for the elements $e^j\otimes e_i$,
\begin{align*}
\Der\g=\operatorname{Span}\biggl\{&e_4^1,e_4^2+e_5^1,e_5^2,-e_1^2+e_2^1-e_4^5+e_5^4,e_6^1,e_6^2,e_4^2+e_6^3,e_3^2+e_4^3+e_6^4,\\
& -e_3^1+e_5^3+e_6^5,\frac{1}{4} e_1^1+\frac{1}{4} e_2^2+\frac{1}{2} e_3^3+\frac{3}{4} e_4^4+\frac{3}{4} e_5^5+e_6^6\biggr\},\\
\lie n=\operatorname{Span}\bigl\{&e_4^1,e_4^2+e_5^1,e_5^2,e_6^1,e_6^2,e_4^2+e_6^3,e_3^2+e_4^3+e_6^4,-e_3^1+e_5^3+e_6^5\bigr\},\\
\lie a_\R=\operatorname{Span}\bigl\{&\frac{1}{4} e_1^1+\frac{1}{4} e_2^2+\frac{1}{2} e_3^3+\frac{3}{4} e_4^4+\frac{3}{4} e_5^5+e_6^6\bigr\},\\
\lie a_{i\R}=\operatorname{Span}\bigl\{&-e_1^2+e_2^1-e_4^5+e_5^4\bigr\}.
\end{align*}
Thus, this Lie algebra admits a unique nontrivial grading,
\[\g_{\frac14}=\Span{e_1,e_2}, \quad \g_{\frac12}=\Span{e_3}, \quad \g_{\frac34}=\Span{e_4,e_5}, \quad \g_1=\Span{e_6}.\]
The weight sequence $\frac14,\frac14,\frac12,\frac12,\frac34,\frac34,1$ satisfies (G1)--(G4): the only way $w_6=1$ can be written as a sum $w_i+w_{\hat i}$ is as $1=w_2+w_5$ or $1=w_4+w_4$. In the first case, $w_2=\frac14$ has multiplicity two; in the second case, $i=\hat i$. It follows that there is an antidiagonal Ricci-flat metric, which in this case takes the form
\[e^2\odot e^6+e^1\odot e^5+e^3\odot e^4.\]
\end{example}
\begin{example}
\label{ex:nonnice6}
A somewhat less illuminating but significant example is the 6-dimensional Lie algebra
\[0,0,0,e^{12},e^{14},e^{15}+e^{23}+e^{24}.\]
In this case $\lie a$ is one-dimensional, generated by
\[\frac{1}{5} e_1^1+\frac{2}{5} e_2^2+\frac{3}{5} e_3^3+\frac{3}{5} e_4^4+\frac{4}{5} e_5^5+e_6^6.\]
The associated weight sequence $\frac{1}{5},\frac{2}{5},\frac{3}{5},\frac{3}{5},\frac{4}{5},1$ satisfies (G1)--(G5), and the antidiagonal metric relative to the basis $e_1,\dotsc, e_6$ is Ricci-flat.

The relevance of this example is that it is the only nilpotent Lie algebra of dimension $6$ which does not admit a nice basis (see \cite{Lauret_Will_2013}); in particular, the arrow-breaking construction of \cite{Conti_delBarco_Rossi_2021} does not apply.
\end{example}

As we observed below Theorem~\ref{thm:foad}, if the maximal split torus $\lie a_\R$ is trivial, or it acts with a weight of $0$, it is not possible to find a grading that satisfies (G1)--(G5). Even if all weights are nonzero, a grading satisfying (G1)--(G5) may fail to exist. We give two examples below.
\begin{example}
On the Lie algebra
\[0,0,0,0,- e^{12},e^{15},e^{25}+e^{34},e^{16},e^{56}+e^{28}+e^{13}\]
a maximal split torus in the space of derivations is given by
\begin{equation*}
 \label{eqn:nontrivialgrading}
 \diag(d_1,\dotsc, d_9)=\diag\bigl(2 \lambda_8-\lambda_9,-3\lambda_8+3\lambda_9,-2 \lambda_8+4 \lambda_9,-2 \lambda_8+ \lambda_9,- \lambda_8+2 \lambda_9,\lambda_8+ \lambda_9,-4 \lambda_8+5 \lambda_9,3\lambda_8,3\lambda_9\bigr).
\end{equation*}
Thus, there is a grading over $\Z\times\Z$ where every basis element $e_i$ has pure degree $d_i$, and one assumes $\lambda_8,\lambda_9$ to be independent generators, and one can obtain gradings over $\Z$ by again declaring that $e_i$ has degree $d_i$ as above, but replacing $\lambda_8,\lambda_9$ with arbitrary integers. A grading satisfying (G1)--(G5) can then be obtained by reordering the $\{e_i\}$ and $\{d_i\}$, say
\[E_1=e_{\sigma_1}, \dotsc, E_9=e_{\sigma_9}, \quad w_1=d_{\sigma_1},\dotsc, w_9=d_{\sigma_9}.\]
By direct inspection, we see that
\[d_1+d_7=d_3, \quad d_3+d_4=d_7.\]
It is therefore impossible to satisfy condition (G1), since we should have $\sigma_3<\sigma_7$ and $\sigma_7<\sigma_3$ simultaneously.
\end{example}
\begin{example}
For the nilpotent Lie algebra
\[0,0,0,- e^{12},e^{14},e^{15}-e^{23},e^{34}-e^{16},e^{35}+e^{17},e^{47}+e^{56}+e^{28}+e^{13},\]
the maximal split torus has dimension one, and there is a unique torsion-free grading, associated to
\[\diag(d_1,\dotsc, d_9)=\diag\left(\frac{1}{4},-\frac{1}{8},\frac{3}{4},\frac{1}{8},\frac{3}{8},\frac{5}{8},\frac{7}{8},\frac{9}{8},1\right).\]
Since $d_2+d_9=d_7, d_4+d_7=d_9$, (G1) cannot be satisfied.
\end{example}

As mentioned above, computing $\lie a$ generally requires determining a Jordan decomposition. However, it turns out that almost all Lie algebras appearing in the classifications of nilpotent Lie algebras of dimension less than and equal to $7$ (contained respectively   in \cite{Magnin_1986} and \cite{Gong_1998}) are written in a basis such that for some maximal torus $\lie a$, $\lie a_{\R}$ acts by diagonal matrices and  $\lie a_{i\R}$ by skew-symmetric matrices. There are a few exceptions; in the Appendix we show how the basis should be changed in these instances to obtain the same.

A more essential reason why we do not need to worry about Jordan decompositions here is a property of nice Lie algebras: despite our construction being general, most of the computations in this paper deal with nice Lie algebras, because classifications in dimensions $8$ and $9$ are only available for the nice case. We recall that a basis $\{e_i\}$ is called a \emph{nice basis} if every Lie bracket $[e_i,e_j]$ is a multiple of some $e_k$ and every interior product $e_i\hook de^j$ is a multiple of some $e^h$ (see \cite{Nikolayevsky_2011, Lauret_Will_2011}); a nice Lie algebra is a Lie algebra endowed with a nice basis; a Lie algebra is non-nice if it does not admit a nice basis. An important property of nice Lie algebras is that the diagonal part of a derivation is a derivation (see \cite{Dere_Lauret_2019}): therefore, diagonal derivations form a canonical split torus, which can be used to determine whether (G1)--(G5) hold, regardless of whether it is maximal.

Having fixed a split torus, or equivalently a grading, determining whether a weight sequence satisfying (G1)--(G5) exists is a matter of iterating through weight sequences. The weight sequences associated to a grading are related to one another by permutations, but it is not necessary to consider all permutations. Indeed, weight sequences can be constructed by picking weights in sequence, and imposing at each step that (G1) is satisfied by requiring that the weight being inserted in the sequence should not precede elements yet to be inserted in the partial ordering defined by (G1).

Since nilpotent Lie algebras of dimension $6$ and nilpotent Lie algebras of dimension $7$ admitting a nice basis have already been shown to admit a Ricci-flat metric in \cite{Conti_delBarco_Rossi_2021}, we apply this method to the remaining  $7$-dimensional nilpotent Lie algebras in the classification of \cite{Gong_1998}.
\begin{theorem}
\label{thm:nonnicegradings}
The non-nice nilpotent Lie algebras of dimension $7$ that admit a grading and weight sequence satisfying (G1)--(G5) are precisely those appearing in Table~\ref{tbl:nonnicefoad}.
\end{theorem}
\begin{proof}
A list of the $39$ nilpotent Lie algebras (including some one-parameter families) of dimension $7$ without a nice basis is given in \cite{Conti_Rossi_2019a}; the Lie algebras appearing in Table~\ref{tbl:nonnicefoad} are a subset. A change of basis is needed in some cases in order to obtain a maximal split torus acting diagonally (see Appendix).

Each entry appearing in Table~\ref{tbl:nonnicefoad} gives a Lie algebra written in terms of a fixed basis $e_1,\dotsc, e_7$, the weights of the action of a maximal split torus $\lie t$ relative to this basis, and a weight sequence satisfying (G1)--(G5). If the torus is one-dimensional, this is a list of numbers; otherwise, if $\lie t$ has a basis $H_1,\dotsc, H_k$, the weights are expressed as linear combinations of the dual basis $\lambda_1,\dotsc, \lambda_k$ of $\lie t^*$.

The non-nice nilpotent Lie algebras of dimension $7$ that do not appear in Table~\ref{tbl:nonnicefoad} fall into two classes: eleven characteristically nilpotent Lie algebras, which cannot satisfy (G1)--(G5), and five Lie algebras that admit a unique nontrivial grading; the latter are given below, together with the weights of $e_1,\dotsc, e_7$  relative to the unique grading:
\[
 \begin{array}{ll}
 \g & \text{grading}\\
 \hline
\rule{0pt}{1.1\normalbaselineskip}
0,0,e^{12},e^{13},e^{14},e^{23},e^{16}+e^{24}+e^{25}-e^{34}&  0,\frac{1}{2},\frac{1}{2},\frac{1}{2},\frac{1}{2},1,1\\
0,0,e^{12},e^{13},e^{23},e^{15}+e^{24},e^{14}+e^{16}+e^{34}&  \frac{1}{3},0,\frac{1}{3},\frac{2}{3},\frac{1}{3},\frac{2}{3},1\\
0,0,e^{12},e^{13},0,e^{14}+e^{25},e^{16}+e^{25}+e^{35}& 0,1,1,1,0,1,1\\
0,0,0,e^{12},e^{14}+e^{23},e^{15}-e^{34},e^{16}+e^{23}-e^{35}&  0,1,0,1,1,1,1\\
0,0,0,e^{12},e^{13},e^{14}+e^{24}-e^{35},e^{25}+e^{34}&  \frac{1}{2},\frac{1}{2},\frac{1}{2},1,1,\frac{3}{2},\frac{3}{2}\\
\end{array}\]
It is easy to verify that, for each of the Lie algebras in the table, (G1)--(G5) cannot be satisfied by any weight sequence: in the first four entries $0$ appears as a weight, so (G1) cannot hold, and in the last entry the only possibile weight sequence satisfying (G1) is the one where the $w_i$ appear in nondecreasing order, which implies $w_3+w_5=w_6$, violating (G2).
\end{proof}

\begin{table}
\caption{\label{tbl:nonnicefoad} Non-nice nilpotent Lie algebras of dimension $7$ admitting a weight sequence satisfying (G1)--(G5), with a weight sequence $w_1,\dotsc, w_7$ and a basis adapted to it.}
\begin{tabular}{LLL}
\lie g  & w_1,\dotsc, w_7& \text{adapted basis}\\
\hline\\
\multicolumn{2}{L}{0,0,e^{12},e^{13},0,e^{14}+e^{23}+e^{25},0}\\ &\frac{2}{5}  \lambda_1,\frac{1}{5}  \lambda_1,\frac{3}{5}  \lambda_1,\frac{3}{5}  \lambda_1,\frac{4}{5}  \lambda_1,\lambda_1,\lambda_2&e_2,e_1,e_3,e_5,e_4,e_6,e_7\\
\multicolumn{2}{L}{0,0,e^{12},e^{13},e^{14}+e^{23},e^{15}+e^{24},e^{23}}\\ &1,\frac{1}{2},\frac{3}{2},2,\frac{5}{2},\frac{5}{2},3&e_2,e_1,e_3,e_4,e_5,e_7,e_6\\
\multicolumn{2}{L}{0,0,e^{12},e^{13},e^{14}+e^{23},e^{25}-e^{34},e^{23}}\\ &\frac{2}{3},\frac{1}{3},1,\frac{4}{3},\frac{5}{3},\frac{5}{3},\frac{7}{3}&e_2,e_1,e_3,e_4,e_5,e_7,e_6\\
\multicolumn{2}{L}{0,0,e^{12},e^{13},e^{14},0,e^{15}+e^{23}+e^{26}}\\ &\frac{1}{7},\frac{3}{7},\frac{4}{7},\frac{4}{7},\frac{5}{7},\frac{6}{7},1&e_1,e_2,e_3,e_6,e_4,e_5,e_7\\
\multicolumn{2}{L}{0,0,e^{12},e^{13},e^{14}+e^{23},0,e^{15}+e^{24}+e^{26}}\\ &\frac{1}{2},1,\frac{3}{2},2,2,\frac{5}{2},3&e_1,e_2,e_3,e_4,e_6,e_5,e_7\\
\multicolumn{2}{L}{0,0,e^{12},e^{13},0,e^{14}+e^{23}+e^{25},e^{16}+e^{24}+e^{35}}\\ &\frac{1}{3},\frac{2}{3},1,1,\frac{4}{3},\frac{5}{3},2&e_1,e_2,e_3,e_5,e_4,e_6,e_7\\
\multicolumn{2}{L}{0,0,e^{12},e^{13},0,e^{14}+e^{23}+e^{25},e^{26}-e^{34}}\\ &1,\frac{1}{2},\frac{3}{2},\frac{3}{2},2,\frac{5}{2},\frac{7}{2}&e_2,e_1,e_3,e_5,e_4,e_6,e_7\\
\multicolumn{2}{L}{0,0,e^{12},e^{13},0,e^{23}+e^{25},e^{14}}\\ &\lambda_1,-\lambda_1+\lambda_2,\lambda_2,\lambda_2,-\lambda_1+2 \lambda_2,\lambda_1+\lambda_2,-2 \lambda_1+3 \lambda_2&e_2,e_1,e_3,e_5,e_4,e_6,e_7\\
\multicolumn{2}{L}{0,0,e^{12},e^{13},0,e^{14}+e^{23},e^{23}+e^{25}}\\ &1,2,3,3,4,5,5&e_1,e_2,e_3,e_5,e_4,e_6,e_7\\
\multicolumn{2}{L}{0,0,e^{12},0,e^{13},e^{23}+e^{24},e^{15}+e^{16}+e^{25}+\lambda e^{26}+e^{34}}\\ &\frac{1}{4},\frac{1}{4},\frac{1}{2},\frac{1}{2},\frac{3}{4},\frac{3}{4},1&e_1,e_2,e_3,e_4,e_5,e_6,e_7\\
\multicolumn{2}{L}{0,0,e^{12},e^{13},0,e^{14}+e^{23}+e^{25},0}\\ &\frac{2}{5}  \lambda_1,\frac{1}{5}  \lambda_1,\frac{3}{5}  \lambda_1,\frac{3}{5}  \lambda_1,\frac{4}{5}  \lambda_1,\lambda_1,\lambda_2&e_2,e_1,e_3,e_5,e_4,e_6,e_7\\
\multicolumn{2}{L}{0,0,e^{12},e^{13},0,e^{14}+e^{23}+e^{25},e^{15}}\\ &1,\frac{1}{2},\frac{3}{2},\frac{3}{2},2,2,\frac{5}{2}&e_2,e_1,e_3,e_5,e_4,e_7,e_6\\
\multicolumn{2}{L}{0,0,0,e^{12},e^{14}+e^{23},e^{23},e^{15}-e^{34}}\\ &-\lambda_1+\lambda_2,2 \lambda_1-\lambda_2,\lambda_1,-2 \lambda_1+2 \lambda_2,\lambda_2,\lambda_2,-\lambda_1+2 \lambda_2&e_1,e_2,e_4,e_3,e_5,e_6,e_7\\
\multicolumn{2}{L}{0,0,e^{12},0,e^{23},e^{14},e^{16}+e^{25}+e^{26}-e^{34}}\\ &\frac{1}{4},\frac{1}{4},\frac{1}{2},\frac{1}{2},\frac{3}{4},\frac{3}{4},1&e_1,e_2,e_3,e_4,e_5,e_6,e_7\\
\multicolumn{2}{L}{0,0,e^{12},0,e^{13}+e^{24},e^{14},e^{15}+e^{23}+\frac{1}{2} e^{26}+\frac{1}{2} e^{34}}\\ &\frac{1}{3},\frac{2}{3},\frac{2}{3},1,1,\frac{4}{3},\frac{5}{3}&e_1,e_2,e_4,e_3,e_6,e_5,e_7\\
\multicolumn{2}{L}{0,0,e^{12},0,e^{13}+e^{24},e^{14},e^{15}+\lambda e^{23}+e^{34}+e^{46}}\\ &\frac{1}{5},\frac{2}{5},\frac{2}{5},\frac{3}{5},\frac{3}{5},\frac{4}{5},1&e_1,e_2,e_4,e_3,e_6,e_5,e_7\\
\multicolumn{2}{L}{0,0,0,e^{12},e^{13},e^{15}+e^{35},e^{25}+e^{34}}\\ &-\lambda_1+\lambda_2,-\lambda_1+\lambda_2,-2 \lambda_1+2 \lambda_2,2 \lambda_1-\lambda_2,\lambda_1,-3 \lambda_1+3 \lambda_2,\lambda_2&e_1,e_3,e_5,e_2,e_4,e_6,e_7\\
\multicolumn{2}{L}{0,0,0,e^{12},e^{23},-e^{13},e^{15}+e^{16}+e^{26}-2 e^{34}}\\ &\lambda_1,-\lambda_1+\lambda_2,-\lambda_1+\lambda_2,\lambda_2,-2 \lambda_1+2 \lambda_2,\lambda_2,-\lambda_1+2 \lambda_2&e_3,e_1,e_2,e_5,e_4,e_6,e_7\\
\multicolumn{2}{L}{0,0,0,e^{12},e^{23},-e^{13},(-e^{16}+e^{25}) \lambda+2 e^{26}-2 e^{34}}\\ & \frac{1}{2} \lambda_2-\frac{1}{2} \lambda_1,\frac{1}{2} \lambda_2-\frac{1}{2} \lambda_1,\lambda_2-\lambda_1,\lambda_1,\frac{1}{2} \lambda_2+\frac{1}{2} \lambda_1,\frac{1}{2} \lambda_2+\frac{1}{2} \lambda_1,\lambda_2&e_1,e_2,e_4,e_3,e_5,e_6,e_7\\
\multicolumn{2}{L}{0,0,0,e^{12},e^{14}+e^{23},0,e^{15}-e^{34}+e^{36}}\\ &\lambda_1,-2 \lambda_1+\lambda_2,2  \lambda_1,-\lambda_1+\lambda_2,-\lambda_1+\lambda_2,\lambda_2,\lambda_1+\lambda_2&e_1,e_2,e_3,e_4,e_6,e_5,e_7\\
\multicolumn{2}{L}{0,0,0,e^{12},e^{14}+e^{23},0,e^{15}+e^{24}-e^{34}+e^{36}}\\ &\frac{1}{5},\frac{2}{5},\frac{2}{5},\frac{3}{5},\frac{3}{5},\frac{4}{5},1&e_1,e_2,e_3,e_4,e_6,e_5,e_7\\
\multicolumn{2}{L}{0,0,e^{12},0,0,e^{13}+e^{14},e^{15}+e^{23}}\\ &2 \lambda_2-2 \lambda_1,\lambda_2-\lambda_1,-\lambda_2+2 \lambda_1,\lambda_1,\lambda_1,\lambda_2,-\lambda_2+3 \lambda_1&e_5,e_2,e_1,e_3,e_4,e_7,e_6\\
\multicolumn{2}{L}{0,0,e^{12},0,0,2 e^{13}+e^{14}+e^{25},e^{15}+2 e^{23}-e^{24}}\\ &\frac{1}{3},\frac{1}{3},\frac{2}{3},\frac{2}{3},\frac{2}{3},1,1&e_1,e_2,e_3,e_4,e_5,e_6,e_7\\
\end{tabular}
\end{table}

The method can also be applied to nice Lie algebras, with the advantage that one can consider the natural split torus formed by diagonal derivations relative to the nice basis,
and that nice nilpotent Lie algebras are classified up to dimension $9$ (see \cite{Conti_Rossi_2019a}). Notice that the classification is up to a notion of isomorphism which is weaker than being isomorphic as Lie algebras, resulting in larger numbers. Some Lie algebras in the classification come in families, but since the properties under study do not depend on the values of the structure constants, we will not distinguish between a nice Lie algebra and a family with the same set of nonvanishing structure constants.

Using this grading, we can determine the nice Lie algebras that admit a weight sequence satisfying (G1)--(G5), obtaining Tables A--D in the ancillary file. We can summarize the conclusion as follows:
\begin{proposition}
\label{prop:nicegraded}
The nice nilpotent Lie algebras which admit a weight sequence associated to the torus of diagonal derivations satisfying (G1)--(G5) are:
\begin{itemize}
 \item all of those of dimension less than $6$;
 \item $34$ out of $36$ in dimension $6$;
 \item $158$ out of $162$ in dimension $7$;
 \item $824$ out of $917$ in dimension $8$;
 \item $5994$ out of $6882$ in dimension $9$.
\end{itemize}
\end{proposition}
The situation can be improved a little by choosing a different split torus. For instance, the $8$-dimensional nice Lie algebra
\[0,0,0,e^{12},e^{13},e^{23},e^{24}+e^{35},e^{36}+e^{14}\]
does not admit a weight sequence satisying (G1)--(G5) relative to the split torus of diagonal derivations, but if one considers the maximal split torus, which acts diagonally on the basis
\[e_1+e_2,\frac12(e_1-e_2),e_3,-e_4,-e_5-e_6,\frac12(e_5-e_6),-e_7-e_8,\frac12(e_7-e_8),\]
the condition is satisfied. In fact, changing the basis in this manner transforms the Lie algebra into
\[0,0,0,e^{12},-e^{13},e^{23},e^{14}+e^{35},e^{24}+e^{36}\]
which is also nice and appears in Table C with a grading satisfying (G1)--(G5). Nevertheless, it is already evident from Theorem~\ref{thm:nonnicegradings} that tweaking bases in this way is not sufficient to find a Ricci-flat metric on every nilpotent Lie algebra.

%

\section{Metrics from filtrations}
\label{sec:enumeratefiltrations}
In this section we illustrate how to apply Theorem~\ref{thm:ffoad} to find a Ricci-flat metric on a fixed Lie algebra. We will start with a Lie algebra expressed in terms of a basis, and look for a filtration with a weight sequence satisfying  (F1)--(F5), assuming that the basis is adapted to the filtration.

In order to construct these filtrations, we proceed as follows:
\begin{enumerate}
\item Find all reorderings of a fixed basis $e_1,\dotsc, e_n$ such that $e_i\hook de^j=0$ only if $i<j$ (reflecting (F1)) and such that $[e_i,e_{\hat i}]\in \Span{e_n}$ (reflecting (F2)).
\item Assign to each $e_i$ a weight $w_i$ and write the conditions \eqref{eqn:filtration} and (F1)--(F5) as a system of linear equalities and inequalities in the unknowns $w_i$.
\item Determine whether the system has a solution, and if so compute one.
\end{enumerate}
The first step amounts to extending a partial order relation to a total order. Finding one such extension is a well-known problem known as topological sorting (see e.g. \cite{Kahn_1962}). In our implementation \cite{skoll}, we used an auxiliary total order on the basis elements given by $e_i<e_j$ when $i<j$, and considered the induced lexicographic order on the set of possible reorderings; we then adapted the algorithm of \cite{Kahn_1962} to iterate in lexicographic order the set of reorderings of $e_1,\dotsc, e_n$ that respect the partial order.

To determine whether a system of linear equalities and inequalities admit a solution, we used a simplistic approach based on the Fourier-Mozkin algorithm. This turned out to be sufficient to handle the low-dimensional Lie algebras considered in this paper, though  a more sophisticated approach using e.g. the simplex algorithm could be necessary to tackle higher dimensions.

As a counterpart to Theorem~\ref{thm:nonnicegradings}, we find:
\begin{theorem}
\label{thm:nonnicefiltration}
Every non-nice nilpotent Lie algebra of dimension $\leq 7$ has a filtration satisfying (F1)--(F5).
\end{theorem}
\begin{proof}
The non-nice nilpotent Lie algebra of dimension $6$ was shown in Example~\ref{ex:nonnice6} to admit a grading satisfying (G1)--(G5). Rescaling the same weight sequence we obtain $1,2,3,3,4,5$, which satisfies (F1)--(F5). An explicit weight sequence is given in Table~\ref{tbl:nonniceffoad} for each non-nice nilpotent Lie algebra of dimension $7$.
\end{proof}

\begin{longtable}[tpc]{L L L}
\caption{\label{tbl:nonniceffoad} Non-nice nilpotent Lie algebras of dimension $7$ and filtrations satisfying (F1)--(F5).}\\
\toprule
\g &   \text{adapted basis}&w_1,\dotsc, w_7  \\
\midrule
\endfirsthead
\multicolumn{3}{c}{\tablename\ \thetable\ -- \textit{Continued from previous page}} \\
\toprule
\g &   \text{adapted basis}&w_1,\dotsc, w_n  \\
\midrule
\endhead
\bottomrule\\[-7pt]
\multicolumn{3}{c}{\tablename\ \thetable\ -- \textit{Continued to next page}} \\
\endfoot
\bottomrule\\[-7pt]
\multicolumn{3}{c}{End of \tablename\ \thetable} \\
\endlastfoot
\multicolumn{2}{L}{0,0,e^{12},e^{13},0,e^{14}+e^{23}+e^{25},0}\\ &e_1,e_2,e_3,e_4,e_5,e_6,e_7&2,3,8,11,12,16,18\\
\multicolumn{2}{L}{0,0,e^{12},e^{13},e^{14},e^{15}+e^{23},e^{16}+e^{23}+e^{24}}\\ &e_1,e_2,e_3,e_4,e_5,e_6,e_7&32,92,144,226,273,320,392\\
\multicolumn{2}{L}{0,0,e^{12},e^{13},e^{14},e^{15}+e^{23},e^{16}+e^{24}+e^{25}-e^{34}}\\ &e_1,e_2,e_3,e_4,e_5,e_6,e_7&48,138,212,315,390,464,576\\
\multicolumn{2}{L}{0,0,e^{12},e^{13},e^{14}+e^{23},e^{15}+e^{24},e^{16}+e^{23}+e^{25}}\\ &e_1,e_2,e_3,e_4,e_5,e_6,e_7&4,13,28,37,46,56,64\\
\multicolumn{2}{L}{0,0,e^{12},e^{13},e^{14}+e^{23},e^{15}+e^{24},-e^{16}+e^{23}-e^{25}}\\ &e_1,e_2,e_3,e_4,e_5,e_6,e_7&2,4,10,13,16,21,24\\
\multicolumn{2}{L}{0,0,e^{12},e^{13},e^{14}+e^{23},e^{15}+e^{24},e^{23}}\\ &e_1,e_2,e_3,e_4,e_5,e_6,e_7&16,32,72,103,128,176,188\\
\multicolumn{2}{L}{0,0,e^{12},e^{13},e^{14}+e^{23},e^{25}-e^{34},e^{23}}\\ &e_1,e_2,e_3,e_4,e_5,e_6,e_7&4,8,16,41,54,64,68\\
\multicolumn{2}{L}{0,0,e^{12},e^{13},e^{14},e^{23},e^{16}+e^{24}+e^{25}-e^{34}}\\ &e_1,e_2,e_3,e_4,e_5,e_6,e_7&1,6,8,14,17,20,24\\
\multicolumn{2}{L}{0,0,e^{12},e^{13},e^{14},e^{23},e^{15}+e^{25}+e^{26}-e^{34}}\\ &e_1,e_2,e_3,e_4,e_5,e_6,e_7&8,8,26,47,63,72,80\\
\multicolumn{2}{L}{0,0,e^{12},e^{13},e^{23},e^{15}+e^{24},e^{14}+e^{16}+e^{25}+e^{34}}\\ &e_1,e_2,e_3,e_4,e_5,e_6,e_7&4,10,16,39,50,56,64\\
\multicolumn{2}{L}{0,0,e^{12},e^{13},e^{23},e^{15}+e^{24},e^{14}+e^{16}-e^{25}+e^{34}}\\ &e_1,e_2,e_3,e_4,e_5,e_6,e_7&4,10,16,39,50,56,64\\
\multicolumn{2}{L}{0,0,e^{12},e^{13},e^{23},e^{15}+e^{24},e^{14}+e^{16}+e^{34}}\\ &e_1,e_2,e_3,e_4,e_5,e_6,e_7&4,10,16,39,50,56,64\\
\multicolumn{2}{L}{0,0,e^{12},e^{13},e^{23},e^{15}+e^{24},e^{14}+e^{16}+\lambda e^{25}+e^{26}+e^{34}-e^{35}}\\ &e_1,e_2,e_3,e_4,e_5,e_6,e_7&2,2,5,8,8,11,13\\
\multicolumn{2}{L}{0,0,e^{12},e^{13},e^{23},-e^{14}-e^{25},e^{16}+e^{25}-e^{35}}\\ &e_1,e_2,e_3,e_4,e_5,e_6,e_7&1,4,8,16,16,22,24\\
\multicolumn{2}{L}{0,0,e^{12},e^{13},e^{23},-e^{14}-e^{25},e^{15}+e^{16}+e^{24}+\lambda e^{25}-e^{35}}\\ &e_1,e_2,e_3,e_4,e_5,e_6,e_7&8,14,31,48,48,68,79\\
\multicolumn{2}{L}{0,0,e^{12},e^{13},e^{14},0,e^{15}+e^{23}+e^{26}}\\ &e_1,e_2,e_3,e_4,e_5,e_6,e_7&8,8,38,51,64,84,92\\
\multicolumn{2}{L}{0,0,e^{12},e^{13},e^{14}+e^{23},0,e^{15}+e^{24}+e^{26}}\\ &e_1,e_2,e_3,e_4,e_5,e_6,e_7&8,8,38,51,64,84,92\\
\multicolumn{2}{L}{0,0,e^{12},e^{13},0,e^{14}+e^{25},e^{16}+e^{25}+e^{35}}\\ &e_1,e_2,e_3,e_4,e_5,e_6,e_7&1,4,8,16,16,22,24\\
\multicolumn{2}{L}{0,0,e^{12},e^{13},0,e^{14}+e^{23}+e^{25},e^{16}+e^{24}+e^{35}}\\ &e_1,e_2,e_3,e_4,e_5,e_6,e_7&1,8,14,16,16,28,30\\
\multicolumn{2}{L}{0,0,e^{12},e^{13},0,e^{14}+e^{23}+e^{25},e^{26}-e^{34}}\\ &e_1,e_2,e_3,e_4,e_5,e_6,e_7&4,4,12,21,26,32,36\\
\multicolumn{2}{L}{0,0,e^{12},e^{13},0,e^{14}+e^{23}+e^{25},e^{15}+e^{26}-e^{34}}\\ &e_1,e_2,e_3,e_4,e_5,e_6,e_7&8,8,26,47,59,72,80\\
\multicolumn{2}{L}{0,0,0,e^{12},e^{14}+e^{23},e^{15}-e^{34},e^{16}+e^{23}-e^{35}}\\ &e_1,e_2,e_3,e_4,e_5,e_6,e_7&1,4,4,12,16,18,20\\
\multicolumn{2}{L}{0,0,e^{12},e^{13},0,e^{23}+e^{25},e^{14}}\\ &e_1,e_2,e_3,e_4,e_5,e_6,e_7&12,21,64,81,86,112,128\\
\multicolumn{2}{L}{0,0,e^{12},e^{13},0,e^{14}+e^{23},e^{23}+e^{25}}\\ &e_1,e_2,e_3,e_4,e_5,e_6,e_7&6,28,40,58,63,80,96\\
\multicolumn{2}{L}{0,0,e^{12},0,e^{13},e^{23}+e^{24},e^{15}+e^{16}+e^{25}+\lambda e^{26}+e^{34}}\\ &e_1,e_2,e_3,e_4,e_5,e_6,e_7&2,2,6,9,12,14,16\\
\multicolumn{2}{L}{0,0,e^{12},e^{13},0,e^{14}+e^{23}+e^{25},0}\\ &e_1,e_2,e_3,e_4,e_5,e_6,e_7&12,56,112,130,133,192,224\\
\multicolumn{2}{L}{0,0,e^{12},e^{13},0,e^{14}+e^{23}+e^{25},e^{15}}\\ &e_1,e_2,e_3,e_4,e_5,e_6,e_7&12,21,64,81,86,112,128\\
\multicolumn{2}{L}{0,0,0,e^{12},e^{14}+e^{23},e^{23},e^{15}-e^{34}}\\ &e_1,e_2,e_3,e_4,e_5,e_6,e_7&2,5,8,11,16,18,20\\
\multicolumn{2}{L}{0,0,e^{12},0,e^{23},e^{14},e^{16}+e^{25}+e^{26}-e^{34}}\\ &e_1,e_2,e_3,e_4,e_5,e_6,e_7&2,2,6,9,12,14,16\\
\multicolumn{2}{L}{0,0,e^{12},0,e^{13}+e^{24},e^{14},e^{15}+e^{23}+\frac{1}{2} e^{26}+\frac{1}{2} e^{34}}\\ &e_1,e_2,e_3,e_4,e_5,e_6,e_7&2,2,6,9,12,14,16\\
\multicolumn{2}{L}{0,0,e^{12},0,e^{13}+e^{24},e^{14},e^{15}+\lambda e^{23}+e^{34}+e^{46}}\\ &e_1,e_2,e_4,e_3,e_5,e_6,e_7&4,4,4,17,24,24,28\\
\multicolumn{2}{L}{0,0,0,e^{12},e^{13},e^{14}+e^{24}-e^{35},e^{25}+e^{34}}\\ &e_2,e_1,e_3,e_4,e_5,e_7,e_6&1,4,4,15,20,22,24\\
\multicolumn{2}{L}{0,0,0,e^{12},e^{13},e^{15}+e^{35},e^{25}+e^{34}}\\ &e_1,e_2,e_3,e_5,e_4,e_6,e_7&4,8,8,21,28,32,36\\
\multicolumn{2}{L}{0,0,0,e^{12},e^{23},-e^{13},e^{15}+e^{16}+e^{26}-2 e^{34}}\\ &e_1,e_2,e_3,e_4,e_5,e_6,e_7&2,2,4,7,9,10,12\\
\multicolumn{2}{L}{0,0,0,e^{12},e^{23},-e^{13},(-e^{16}+e^{25}) \lambda+2 e^{26}-2 e^{34}}\\ & e_1,e_2,e_3,e_4,e_5,e_6,e_7&2,2,4,7,9,10,12\\
\multicolumn{2}{L}{0,0,0,e^{12},e^{14}+e^{23},0,e^{15}-e^{34}+e^{36}}\\ &e_1,e_2,e_3,e_4,e_5,e_6,e_7&4,4,4,11,16,16,20\\
\multicolumn{2}{L}{0,0,0,e^{12},e^{14}+e^{23},0,e^{15}+e^{24}-e^{34}+e^{36}}\\ &e_1,e_2,e_3,e_4,e_5,e_6,e_7&4,4,4,17,24,24,28\\
\multicolumn{2}{L}{0,0,e^{12},0,0,e^{13}+e^{14},e^{15}+e^{23}}\\ &e_1,e_2,e_3,e_4,e_5,e_6,e_7&8,15,28,34,39,48,58\\
\multicolumn{2}{L}{0,0,e^{12},0,0,2 e^{13}+e^{14}+e^{25},e^{15}+2 e^{23}-e^{24}}\\ &e_1,e_2,e_3,e_4,e_5,e_6,e_7&4,8,16,17,18,28,32\\
\end{longtable}

The procedure we have illustrated does not determine whether a given Lie algebra admits a filtration satisfying (F1)--(F5), because it only detects filtrations adapted to a fixed basis. In order to show that a filtration satisfying (F1)--(F5) does not exist, we can resort to the following:
\begin{lemma}
\label{lemma:nofiltration}
Let $\g$ be a filiform nilpotent Lie algebra of dimension $2n$; let $e_1,\dotsc, e_{2n}$ be a basis adapted to the lower central series in the sense that $\g^k$ is spanned by $e_{2+k},\dotsc, e_{2n}$. If  $[e_k,e_{\hat k}]$ is not zero for some $3\leq k\leq n$, then there is no filtration satisfying (F1)--(F5).
\end{lemma}
\begin{proof}
Suppose for a contradiction that a filtration satisfying (F1)--(F5) exists. Let $\alpha_i$ be the weight of the smallest layer containing $e_i$. Then $[e_i,e_j]=e_k$ implies $\alpha_i+\alpha_j\leq\alpha_k$. Thus, we have
\[0<\alpha_1,\alpha_2<\alpha_3<\dots <\alpha_{2n}.\]
Each $\alpha_i$ is a weight. The layer with weight $\alpha_{2n}$ only contains $e_{2n}$; therefore, the weight sequence $w_1,\dotsc, w_{2n}$ terminates at $w_{2n}=\alpha_{2n}$, and by the same argument we obtain
\[w_3=\alpha_3,\dotsc, w_{2n}=\alpha_{2n}.\]
Therefore $e_k$ has weight $w_k$ and $e_{\hat k}$ has weight $w_{\hat k}>w_{k}$; both have multiplicity one, and since $[e_k,e_{\hat k}]\neq0$ we have $w_k+w_{\hat k}\leq w_{2n}$, contradicting (F3).
\end{proof}

\begin{proposition}
\label{prop:filtrationupto8}
Every nice Lie algebra of dimension $\leq 8$ has a filtration satisfying (F1)--(F5) except the following, which do not admit one:
\begin{gather*}
0,0,-e^{12},e^{13},e^{14},e^{25}+e^{34}\\
0,0,-e^{12},e^{13},e^{14}+e^{23},e^{25}+e^{34}\\
0,0,-e^{12},-e^{13},e^{14},e^{15},e^{16},e^{27}+e^{36}+e^{45}\\
0,0,-e^{12},-e^{13},e^{14},e^{15},e^{16}+e^{23},e^{27}+e^{36}+e^{45}\\
0,0,-e^{12},e^{13},e^{14},-e^{15}+e^{23},e^{16}+e^{24},e^{27}+e^{36}+e^{45}\\
0,0,-e^{12},-e^{13},\frac{3}{2} e^{14}-\frac{1}{2} e^{23},e^{15}+\frac{1}{2} e^{24},e^{16}+e^{25}+e^{34},e^{27}+e^{36}+e^{45}\\
0,0,0,-e^{12},e^{13},e^{15}+e^{24},e^{14}+e^{35},e^{26}+e^{37}+e^{45}
\end{gather*}
\end{proposition}
\begin{proof}
We first show that the Lie algebras appearing in the statement do not admit a filtration satisfying (F1)--(F5). The first six entries are filiform of dimension $2n=6,8$ with $[e_n,e_{n+1}]$ nonzero, so we can apply directly Lemma~\ref{lemma:nofiltration}.

For the last one, suppose that a weight sequence $w_1,\dotsc, w_9$ satisfies (F1)--(F5). By a dimension count, $\g'$ is the layer of weight $w_4$, and $\g''$ is the layer of weight $w_6$. In particular, $w_4,w_5$ are strictly contained between $w_3$ and $w_6$. Since $\g'=\g''\oplus\Span{e_4,e_5}$ and $[e_4,e_5]\neq0$, we see that $w_4+w_5\leq w_{9}$, which violates (F3) or (F4).

For all other nice nilpotent Lie algebra of dimension $\leq8$, an explicit filtration and weight sequence satisfying (F1)--(F5) is given in Tables~E, F and G (see ancillary file).
\end{proof}

Considering that nice nilpotent Lie algebras of dimension $8$ are 917 in number, Proposition~\ref{prop:filtrationupto8} shows that the filtration ansatz is effective, but it does not quite solve the problem of constructing a Ricci-flat metric on every nilpotent Lie algebra. The situation becomes worse in dimension $9$, where even proving nonexistence of the filtration is generally harder, as we can see in the following example:
\begin{example}
The nilpotent Lie algebra
\[0,0,0,- e^{12},e^{14},e^{15}-e^{23},e^{34}-e^{16},e^{35}+e^{17},e^{47}+e^{56}+e^{28}+e^{13}.\]
does not admit a filtration satisfying (F1)--(F5).

Indeed, consider an arbitrary positive filtration. As in the proof of Lemma~\ref{lemma:nofiltration}, let $\alpha_i$ be the weight of the smallest layer containing $e_i$. Thus, we have
\[\alpha_1,\alpha_2<\alpha_4<\alpha_5, \quad \alpha_3,\alpha_5<\alpha_6<\alpha_7<\alpha_8<\alpha_9.\]
Arguing as in  Lemma~\ref{lemma:nofiltration}, we see that
\[w_k=\alpha_k, \quad k=6,7,8,9.\]
Since the filtration is positive, $e_4$ belongs to a layer contained in $\Span{e_3,e_4,\dotsc, e_9}$. Thus $\alpha_4$ coincides with one of $w_3,w_4,w_5$; however, the possibility $\alpha_4=w_5$ cannot occur, for otherwise $w_5<\alpha_5<w_6$.

Now assume that (F1)--(F5) hold. If $\alpha_4=w_3<w_4$, then $w_1,w_2<w_3$, so $w_3$ has multiplicity one, as does $w_7$. This violates (F3), because  $[e_4,e_7]=-e_9$ and (F2) implies $w_3+w_7=w_9$.

Thus, $\alpha_4=w_4$ and $\alpha_5=w_5$. Arguing as above, we see that $w_3\neq\alpha_1$, since $[e_1,e_7]=-e_8$, but $w_3+w_7$ should exceed $w_8$ by (F2).  Thus, $\alpha_1$ equals $w_1$ or $w_2$.

We have
\[w_8\geq \alpha_1+\alpha_7\geq 2\alpha_1+\alpha_6.\]
Thus, (F2) gives
\[w_4+w_6> w_8\geq 2\alpha_1+\alpha_6,\]
and $w_4>2\alpha_1$. Since $\alpha_1+\alpha_2\leq w_4$, this implies that $\alpha_1<\alpha_2$. Therefore $\alpha_2$ equals $w_2$ or $w_3$. If $\alpha_2=w_2<w_3$, then it has multiplicity one, which contradicts the fact that $[e_2,e_8]=-e_9$. Thus, necessarily $\alpha_2=w_3$, implying that $w_4$ has multiplicity one. Therefore, $w_4+w_6>w_9$ by (F3), and therefore
\[w_5+w_6>w_4+w_6>w_9,\]
which contradicts \eqref{eqn:filtration} because $[e_5,e_6]=-e_9$.
\end{example}

\section{Proofs of Theorems A and B}
In this section we prove that nilpotent Lie algebras of dimension $\leq7$ and nice nilpotent Lie algebras of dimension $\leq 9$ have a Ricci-flat metric; these results were stated as Theorems A and B in the introduction.

Combining Theorem~\ref{thm:nonnicefiltration}, Proposition~\ref{prop:filtrationupto8} and Example~\ref{ex:nonnice6}, we see that every nilpotent Lie algebra of dimension $\leq 7$ and every nice nilpotent Lie algebra of dimension $\leq8$ have a Ricci-flat metric except those appearing in Proposition~\ref{prop:filtrationupto8}. For the others, a different ansatz is needed. The two entries of dimension $6$ are already covered by the arrow-breaking construction of \cite{Conti_delBarco_Rossi_2021}, which we will not recall here, since it does not apply to all the examples we need to consider. Instead, we consider the more general class of $\sigma$-diagonal metrics. Recall from  \cite{Conti_Rossi_2019b} that a $\sigma$-diagonal metric on a nice Lie algebra $\g$ with nice basis $e_1,\dotsc, e_n$ is a metric of the form $\sum_{i=1}^ng_ie^i\otimes e^{\sigma_i}$, with $\sigma$ an order two permutation of the indices $\{1,\dotsc, n\}$ and the $g_i$ nonzero real numbers such that $g_i=g_{\sigma_i}$.  A $\sigma$-diagonal metric is determined by the permutation $\sigma$ and a $\sigma$-invariant sequence of nonzero elements $(g_1,\dotsc, g_n)$.

For the $8$-dimensional entries listed in Proposition~\ref{prop:filtrationupto8}, we can construct an explicit $\sigma$-diagonal Ricci-flat metric; see Table~\ref{table:remaining8}. Notice that in all cases except one, the parameters $g_i$ can be chosen arbitrarily.

In order to complete the proof of  Theorem B, we need to show that every nice nilpotent Lie algebra of dimension $9$ has a Ricci-flat metric. We noted in Proposition~\ref{prop:nicegraded} that $5994$ nice nilpotent Lie algebras of dimension $9$ admit a grading satisfying (G1)--(G5). Computations with \cite{skoll} show that  $857$ of the remaining Lie algebras admit a filtration and a weight sequence satifying (F1)--(F5), and  the $31$ remaining Lie algebras carry a $\sigma$-diagonal metric which is Ricci-flat for any choice of  the parameters $g_i$.  The gradings, filtrations and $\sigma$-diagonal metrics are given explicitly  in Table H in the ancillary file.

\begin{table}
\caption{\label{table:remaining8}$\sigma$-diagonal Ricci-flat metrics on nice nilpotent Lie algebras of dimension $8$ without a filtration satisfying (F1)--(F5)}
\[
\begin{array}{ll}
\g&\sigma\\
\hline
\rule{0pt}{1.1\normalbaselineskip}
0,0,-e^{12},-e^{13},e^{14},e^{15},e^{16},e^{27}+e^{36}+e^{45}&18\ 35\
\\
0,0,-e^{12},-e^{13},e^{14},e^{15},e^{16}+e^{23},e^{27}+e^{36}+e^{45}&17\ 28\ 35\
\\
0,0,-e^{12},e^{13},e^{14},-e^{15}+e^{23},e^{16}+e^{24},e^{27}+e^{36}+e^{45}&18\ 26\ 37\
\\
0,0,-e^{12},-e^{13},\frac{3}{2} e^{14}-\frac{1}{2} e^{23},e^{15}+\frac{1}{2} e^{24},e^{16}+e^{25}+e^{34},e^{27}+e^{36}+e^{45}&18\ 26\ 37\\
\multicolumn{2}{r}{\text{with }8g_4g_5 (g_3^{2} -  g_1^{2} )+g_2g_3(9 g_5^{2} +4 g_1^{2} )=0}\\
0,0,0,-e^{12},e^{13},e^{15}+e^{24},e^{14}+e^{35},e^{26}+e^{37}+e^{45}&47\ 56\
\\
\end{array}\]
\end{table}

\section*{Appendix}
In this appendix we show how the  nilpotent Lie algebras of dimension $6,7$ classified respectively in \cite{Magnin_1986} and \cite{Gong_1998} can be rewritten so that a maximal split torus $\lie a$ appears as the direct sum of a space of diagonal matrices and a space of skew-symmetric matrices.
\subsection*{Dimension $6$}
For the Lie algebra
\[0,0,e^{12},e^{13},e^{23},e^{14}-e^{25},\]
we see that the maximal split torus is given by
\[\lie a_\R=\Span{e_1^2+e_2^1+e_4^5+e_5^4,\frac{1}{4} e_1^1+\frac{1}{4} e_2^2+\frac{1}{2} e_3^3+\frac{3}{4} e_4^4+\frac{3}{4} e_5^5+e_6^6}.\]
Relative to the frame $E_1=e_1+e_2, E_2=e_1-e_2, E_3=-2e_3, E_4=-2(e_4+e_5), E_5=-2(e_4-e_5), -4E_6$, we obtain the nice Lie algebra
\[0,0, E^{12},E^{13},E^{23},E^{15}+E^{24},\]
on which $\lie a_\R$ acts diagonally. Similarly, for
\[0,0,0,e^{12},e^{23},e^{14}+e^{35}\]
we compute
\[a_\R=\Span{e_1^1-2 e_2^2+e_3^3-e_4^4-e_5^5,-e_1^3-e_3^1+e_4^5+e_5^4,e_2^2+e_4^4+e_5^5+e_6^6},\]
so changing the frame to $(e_2,e_1+e_3,e_1-e_3,-e_4+e_5,-e_4-e_5,-2e_6)$, we obtain the nice Lie algebra
\[0,0,0,e^{12},e^{13},e^{25}+e^{34}.\]
\subsection*{Dimension 7}
For dimension 7, beside the products with $\R$ of the two six-dimensional cases, the Lie algebras of \cite{Gong_1998} for which a change of basis is needed are the following:
\[247F: 0,0,0,e^{12},e^{13},e^{24}+e^{35},e^{25}+e^{34}\rightsquigarrow 0,0,0,e^{12},e^{13},e^{24},e^{35},\]
having changed the frame to $(e_1,e_2+e_3,e_2-e_3,e_4+e_5,e_4-e_5,2(e_6+e_7),2(e_6-e_7))$, with the label $247F$ referring to the classification of \cite{Gong_1998};
\[247G: 0,0,0,e^{12},e^{13},e^{14}+e^{15}+e^{24}+e^{35},e^{25}+e^{34}\rightsquigarrow 0,0,0,e^{12},e^{13},e^{24},e^{14}+e^{35},\]
having changed the frame to $(2e_1-e_2-e_3,e_2+e_3,e_2-e_3,2(e_4+e_5),2(e_4-e_5),4(e_6+e_7),4(e_6-e_7))$;
\[257J_1: 0,0,e^{12},0,0,e^{13}+e^{14}+e^{25},e^{15}+e^{23}\rightsquigarrow 0,0,e^{12},0,0,2 e^{13}+e^{14}+e^{25},e^{15}+2 e^{23}-e^{24},\]
having changed the frame to $(e_1,e_2,e_3,e_4-\frac12e_3,
\frac12e_5,\frac12e_6,\frac12e_7)$;
\[2457L: 0,0,e^{12},e^{13},e^{23},e^{14}+e^{25},e^{15}+e^{24}\rightsquigarrow 0,0,e^{12},e^{13},e^{23},e^{14},e^{25},\]
having changed the frame to $\frac12(e_1+e_2),\frac12(e_2-e_1),\frac12e_3,\frac14(e_4+e_5),\frac14(e_5-e_4),\frac14(e_6+e_7),\frac14(e_6-e_7)$;
and lastly the one-parameter family
\[147E1: 0,0,0,e^{12},e^{23},- e^{13}, \lambda {(e^{25}-e^{16})}+2 e^{26}-2 e^{34}, \quad \lambda>1.\]
In this last case,  $\lie a$ is spanned by
\[-e_1^1+e_2^2+e_5^5-e_6^6+ \lambda(-e_6^5 + e_5^6 - e_2^1 + e_1^2),-e_1^1+2e_2^2-e_3^3+e_5^5-2e_6^6, e_3^3+e_5^5+e_6^6+e_7^7,\]
where the first element acts by imaginary eigenvalues and the other two by real eigenvalues. We will not change the basis for this Lie algebra because the resulting structure constants are considerably more complicated, and the computation of Theorem~\ref{thm:nonnicegradings} is not affected, as it only uses $\lie a_\R$.

We also note that the Lie algebra
\[12457B: 0,0,e^{12},e^{13},0,e^{14}+e^{25},e^{25}+e^{35}+e^{16}\]
appears with a typo in \cite[Table 2]{Conti_Rossi_2019a}.

\smallskip
\noindent \textbf{Funding details:}
The author acknowledges: the MIUR Excellence Department Project awarded to the Department of Mathematics, University of Pisa, CUP I57G22000700001; the PRIN project n. 2022MWPMAB ``Interactions between Geometric Structures and Function Theories'';  GNSAGA of INdAM.

\printbibliography

\small\noindent   Dipartimento di Matematica, Università di Pisa, largo Bruno Pontecorvo 6, 56127 Pisa, Italy.\\
\texttt{diego.conti@unipi.it}

\end{document}